\documentclass[11pt, twoside, leqno]{article}

\usepackage{amssymb}
\usepackage{amsmath}
\usepackage{amsthm}
\usepackage{color}
\usepackage{mathrsfs}
\usepackage{txfonts}

\usepackage{indentfirst}

\allowdisplaybreaks

\def\rr{{\mathbb R}}
\def\rn{{\mathbb{R}^n}}
\def\zz{{\mathbb Z}}
\def\nn{{\mathbb N}}
\def\cp{{\mathcal P}}
\def\cs{{\mathcal S}}

\def\fz{\infty }
\def\az{\alpha}
\def\lz{\lambda}
\def\pa{\partial}
\def\f{\frac}

\def\lf{\left}
\def\r{\right}

\def\ls{\lesssim}
\def\noz{\nonumber}
\def\wz{\widetilde}
\def\st{\subset}

\def\rg{\rangle}
\def\supp{\mathop\mathrm{\,supp\,}}

\def\va{\vec{a}}
\def\vp{\vec{p}}
\def\vh{{H_{\va}^{\vp}(\rn)}}
\def\vah{{H_{\va}^{\vp,\,r,\,s}(\rn)}}
\def\vfah{{H_{\va,\,{\rm fin}}^{\vp,\,r,\,s}(\rn)}}
\def\vfahfz{{H_{\va,\,{\rm fin}}^{\vp,\,\fz,\,s}(\rn)}}
\def\lv{{L^{\vp}(\rn)}}

\def\lq{\mathcal{L}_{\vp,\,q,\,s}^{\va}(\rn)}
\def\lr{\mathcal{L}_{\vp,\,r',\,s}^{\va}(\rn)}

\newtheorem{theorem}{Theorem}[section]
\newtheorem{lemma}[theorem]{Lemma}
\newtheorem{corollary}[theorem]{Corollary}

\theoremstyle{definition}
\newtheorem{remark}[theorem]{Remark}
\newtheorem{definition}[theorem]{Definition}
\renewcommand{\appendix}{\par
   \setcounter{section}{0}%
   \setcounter{subsection}{0}%
   \setcounter{subsubsection}{0}%
   \gdef\thesection{\@Alph\c@section}%
   \gdef\thesubsection{\@Alph\c@section.\@arabic\c@subsection}%
   \gdef\theHsection{\@Alph\c@section.}%
   \gdef\theHsubsection{\@Alph\c@section.\@arabic\c@subsection}%
   \csname appendixmore\endcsname
 }

\numberwithin{equation}{section}

\begin{document}

\arraycolsep=1pt

\title{\bf\Large Dual Spaces of Anisotropic Mixed-Norm Hardy Spaces
\footnotetext{\hspace{-0.35cm} 2010 {\it
Mathematics Subject Classification}. Primary 42B35;
Secondary 42B30, 46E30.
\endgraf {\it Key words and phrases.}
anisotropic Euclidean space, (mixed-norm) Campanato space, (mixed-norm) Hardy space,
duality.
\endgraf This project is supported by the National
Natural Science Foundation of China
(Grant Nos.~11761131002, 11571039, 11726621 and 11471042)
and also by the Joint Research Project Between China Scholarship Council and German
Academic Exchange Service (PPP) (Grant No.~LiuJinOu [2016]6052).}}
\author{Long Huang, Jun Liu, Dachun Yang\footnote{Corresponding author / March 26, 2018 / Newest Version.}\ \
and Wen Yuan}
\date{}
\maketitle

\vspace{-0.8cm}

\begin{center}
\begin{minipage}{13cm}
{\small {\bf Abstract}\quad
Let $\vec{a}:=(a_1,\ldots,a_n)\in[1,\infty)^n$,
$\vec{p}:=(p_1,\ldots,p_n)\in(0,\infty)^n$
and $H_{\vec{a}}^{\vec{p}}(\mathbb{R}^n)$
be the anisotropic mixed-norm Hardy space associated with $\vec{a}$
defined via the non-tangential grand maximal function. In this article,
the authors give the dual space of
$H_{\vec{a}}^{\vec{p}}(\mathbb{R}^n)$, which was asked
by Cleanthous et al. in [J. Geom. Anal. 27 (2017), 2758-2787].
More precisely,
via first introducing the anisotropic mixed-norm Campanato space
$\mathcal{L}_{\vec{p},\,q,\,s}^{\vec{a}}(\mathbb{R}^n)$ with $q\in[1,\infty]$ and
$s\in\mathbb{Z}_+:=\{0,1,\ldots\}$,
and applying the known atomic and finite atomic characterizations of
$H_{\vec{a}}^{\vec{p}}(\mathbb{R}^n)$, the authors prove
that the dual space of $H_{\vec{a}}^{\vec{p}}(\mathbb{R}^n)$ is
the space $\mathcal{L}_{\vec{p},\,r',\,s}^{\vec{a}}(\mathbb{R}^n)$ with
$\vec{p}\in(0,1]^n$, $r\in(1,\infty]$, $1/r+1/r'=1$ and
$s\in[\lfloor\frac{\nu}{a_-}(\frac{1}{p_-}-1) \rfloor,\infty)\cap\mathbb{Z}_+$,
where $\nu:=a_1+\cdots+a_n$, $a_-:=\min\{a_1,\ldots,a_n\}$,
$p_-:=\min\{p_1,\ldots,p_n\}$ and,
for any $t\in \mathbb{R}$, $\lfloor t\rfloor$ denotes the largest integer
not greater than $t$. This duality result is new even for
the isotropic mixed-norm Hardy spaces on $\mathbb{R}^n$.
}
\end{minipage}
\end{center}

\section{Introduction\label{s1}}

The main purpose of this article is to give the dual space
of the anisotropic mixed-norm Hardy space on $\rn$. Recall that,
as a generalization of the classical Hardy space $H^p(\rn)$,
the anisotropic mixed-norm Hardy space $\vh$, in which the constant
exponent $p\in(0,\fz)$ is replaced by an exponent vector $\vp\in (0,\fz)^n$ and
the Euclidean norm $|\cdot|$ on $\rn$ by the anisotropic homogeneous
quasi-norm $|\cdot|_{\va}$ with $\va\in[1,\fz)^n$ (see Definition \ref{2d1} below),
was first considered by Cleanthous et al. in \cite{cgn17}.
Cleanthous et al. \cite{cgn17} introduced the anisotropic mixed-norm Hardy space $\vh$
with $\va\in [1,\fz)^n$ and $\vp\in (0,\fz)^n$ via the non-tangential grand maximal
function and investigated its radial or its non-tangential maximal function characterizations.
In particular, they mentioned several natural questions to be studied
(see \cite[p.\,2760]{cgn17}), which include the atomic characterizations
and the duality theory of $\vh$ as well as the boundedness of
anisotropic singular integral operators on these Hardy-type spaces.
To answer these questions and also to complete the real-variable theory of
the anisotropic mixed-norm Hardy space $\vh$, Huang et al. \cite{hlyy} established
several equivalent characterizations of $\vh$, respectively, in terms of the atom,
the finite atom, the Lusin area function, the Littlewood-Paley $g$-function or
$g_{\lambda}^\ast$-function and also obtained
the boundedness of anisotropic Calder\'{o}n-Zygmund operators on $\vh$.
However, the aforementioned question on duality theory of $\vh$ is still missing so far.
In addition, the theory of anisotropic or mixed-norm function spaces was
developed well in recent years; see, for example,
\cite{cs,cgn17bs,cgn17-2,htw17,jms15}.

It is well known that the duality theory of classical Hardy spaces on the
Euclidean space $\rn$ plays an important role in many branches of analysis such
as harmonic analysis and partial differential equations, and has been systematically
considered and developed; see, for example, \cite{fs72,m94,s93}.
In 1969, Duren et al. \cite{drs69} first showed that the dual space of the Hardy space
$H^p(\mathbb{D})$ of holomorphic functions is the Lipshitz space,
where $p\in(0,1)$ and the \emph{symbol} $\mathbb{D}$ denotes the unit disc of $\rn$.
Later on, Walsh \cite{w73} further extended this duality result to the Hardy
space on the upper half-plane $\mathbb{R}^{n+1}_+$ with $p\in(0,1)$. Moreover,
the prominent duality theory, namely, the bounded mean oscillation function space
$\mathop{\mathrm{BMO}}(\rn)$ is the dual space of the Hardy space $H^1(\rn)$ is due
to Fefferman and Stein \cite{fs72}. It is worth to point out that the complete duality
theory of the classical Hardy space $H^p(\rn)$, with $p\in(0,1]$, is given by Taibleson
and Weiss \cite{tw80}, in which the dual space of $H^p(\rn)$ was proved to be
the Campanato space introduced by Campanato \cite{c64}. We should also point out that,
nowadays, the theory related to
Campanato spaces has been developed well and proved useful in many areas of analysis;
see, for example, \cite{jxy16,ky17,llw16,ly13,ns12,ylk17,ysy}. In addition,
based on the duality results of the classical Hardy spaces mentioned as above
(see \cite{drs69,fs72,w73}) as well as the celebrated work of Calder\'{o}n and
Torchinsky \cite{ct75} on the parabolic Hardy space, Calder\'{o}n and Torchinsky \cite{ct77}
further studied the duality theory of the parabolic Hardy space. For more developments
of the duality theory of function spaces and their applications in harmonic analysis and
partial differential equations, we refer the reader to
\cite{mb03,dll17,lby14,mmv15,ns12,ylk17,ysy,zsy16}.

Notice that, when $\va\in[1,\fz)^n$ and
$\vp:=(p_1,\ldots,p_n)\in(1,\fz)^n$, by \cite[Theorem 6.1]{cgn17},
we know that $\vh=\lv$ with equivalent quasi-norms, which, together with the known fact
that the dual of $\lv$ is $L^{\vp'}(\rn)$ (see \cite[p.\,304, Theorem 1.a)]{bp61}), where $\vp':=(p_1',\ldots,p_n')$ and,
for any $i\in\{1,\ldots,n\}$, $1/p_i+1/p_i'=1$, implies that, for any $\vp\in(1,\fz)^n$,
$L^{\vp'}(\rn)$ is the dual space of $\vh$.
In this article, we further complete the duality theory of
$H_{\vec{a}}^{\vec{p}}(\mathbb{R}^n)$, which partly answers the aforementioned question
of Cleanthous et al. in \cite[p.\,2760]{cgn17} on the duality theory. Precisely, let
$\vec{a}:=(a_1,\ldots,a_n)\in[1,\infty)^n$, $\vec{p}:=(p_1,\ldots,p_n)\in(0,\infty)^n$,
$q\in[1,\fz]$ and $s\in\mathbb{Z}_+:=\{0,1,\ldots\}$, we first introduce the anisotropic
mixed-norm Campanato space $\lq$.
Then, applying the known atomic and finite atomic characterizations of
$H_{\vec{a}}^{\vec{p}}(\mathbb{R}^n)$ obtained in \cite{hlyy}
(see Lemmas \ref{3l1} and \ref{3l2} below),
we prove that the dual space of $H_{\vec{a}}^{\vec{p}}(\mathbb{R}^n)$ is
the space $\mathcal{L}_{\vec{p},\,r',\,s}^{\vec{a}}(\mathbb{R}^n)$ with
$\vec{p}\in(0,1]^n$, $r\in(1,\fz]$, $1/r+1/r'=1$ and
$s\in[\lfloor\frac{\nu}{a_-}(\frac{1}{p_-}-1) \rfloor,\infty)\cap\mathbb{Z}_+$,
where $\nu:=a_1+\cdots+a_n$, $a_-:=\min\{a_1,\ldots,a_n\}$,
$p_-:=\min\{p_1,\ldots,p_n\}$ and, for any $t\in \mathbb{R}$,
the \emph{symbol} $\lfloor t\rfloor$ denotes the largest integer
not greater than $t$. This duality result is new even for the isotropic mixed-norm
Hardy spaces on $\mathbb{R}^n$. We should point out that, when
$\vp:=(p_1,\ldots,p_n)\in(0,\fz)^n$ with $p_{i_0}\in(0,1]$ and $p_{j_0}\in(1,\fz)$
for some $i_0,j_0\in\{1,\ldots,n\}$, the dual space of $\vh$ is still unknown so far.

Concretely, this article is organized as follows.

In Section \ref{s2}, we first recall some notions and notation appearing in this article,
including the anisotropic homogeneous quasi-norm, the anisotropic bracket and
the mixed-norm Lebesgue space. Then we present the definition of the anisotropic mixed-norm
Hardy spaces $\vh$ via the non-tangential grand maximal functions from \cite{cgn17}
(see Definition \ref{2d5} below).

Section \ref{s3} is devoted to establishing the duality theory of $\vh$ with
$\va\in [1,\fz)^n$ and $\vp\in (0,1]^n$.
To this end, we first introduce the anisotropic mixed-norm Campanato space $\lq$
(see Definition \ref{3d1'} below), which includes the space $\mathop{\mathrm{BMO}}(\rn)$ of
John and Nirenberg \cite{jn61} as well as the classical Campanato spaces of
Campanato \cite{c64} as special cases [see Remark \ref{3r1}(ii) below].
Then, via borrowing some ideas from \cite[Theorem 3.5]{ly13} and
\cite[p.\,51, Theorem 8.3]{mb03}, we prove that the dual space of $\vh$ is the space
$\lr$ with $r\in(1,\fz]$, $1/r+1/r'=1$ and $s$ being as in \eqref{3e1} below
(see Theorem \ref{3t1} below). To be precise, by the known atomic and finite atomic
characterizations of $\vh$ (see Lemmas \ref{3l1} and \ref{3l2} below) as well as
an argument similar to that used in the proof of \cite[Theorem 3.5]{ly13}
(see also \cite[Theorem 5.2.1]{ylk17}), we show that the anisotropic mixed-norm
Campanato space $\lr$ is continuously embedded into $[\vh]^*$ with
$r$ and $s$ as in Theorem \ref{3t1} below, where the \emph{symbol} $[\vh]^*$ denotes
the dual space of $\vh$.
Conversely, to prove $[\vh]^*\st\lr$ and the inclusion is continuous, motivated by
\cite[Lemma 5.9]{zsy16} and \cite[p.\,51, Lemma 8.2]{mb03}, we first establish two
useful estimates (see, respectively, Lemmas \ref{3l3} and \ref{3l4} below), which
play a key role in the proof of Theorem \ref{3t1} and are also of independent interest.
Via these two lemmas, the atomic characterizations of $\vh$ again and the Hahn-Banach
theorem (see, for example, \cite[Theorem 3.6]{ru91}) as well as a proof similar to that
of \cite[p.\,51, Theorem 8.3]{mb03}, we then show that $[\vh]^*$ is continuously
embedded into $\lr$, which then completes the proof of Theorem \ref{3t1}.

Finally, we make some conventions on notation.
We always let
$\mathbb{N}:=\{1,2,\ldots\}$,
$\mathbb{Z}_+:=\{0\}\cup\mathbb{N}$
and $\vec0_n$ be the \emph{origin} of $\rn$.
For any multi-index
$\az:=(\az_1,\ldots,\az_n)\in(\mathbb{Z}_+)^n=:\mathbb{Z}_+^n$,
let $|\az|:=\az_1+\cdots+\az_n$ and
$\pa^{\az}:=(\f{\pa}{\pa x_1})^{\az_1} \cdots (\f{\pa}{\pa x_n})^{\az_n}$.
We denote by $C$ a \emph{positive constant}
which is independent of the main parameters,
but may vary from line to line. If $f\le Cg$, then we write $f\ls g$
for simplicity, and the \emph{symbol} $f\sim g$ means $f\ls g\ls f$.
For any $r\in[1,\fz]$, the notation $r'$ denotes
its \emph{conjugate index}, namely, $1/r+1/r'=1$.
Moreover, if $\vec{r}:=(r_1,\ldots,r_n)\in[1,\fz]^n$, we denote by
$\vec{r}':=(r_1',\ldots,r_n')$ its \emph{conjugate index}, namely,
for any $i\in\{1,\ldots,n\}$, $1/r_i+1/r_i'=1$. In addition,
for any set $F\subset\rn$, we denote by $F^\complement$ the
set $\rn\setminus F$, by $\chi_F$ its \emph{characteristic function}
and by $|F|$ its \emph{n-dimensional Lebesgue measure}.
For any $t\in\mathbb{R}$, the \emph{symbol} $\lfloor t\rfloor$ denotes
the \emph{largest integer not greater than $t$}. In what follows,
we denote by $C^{\fz}(\rn)$ the set of all
\emph{infinitely differentiable functions} on $\rn$.

\section{Preliminaries \label{s2}}

In this section, we recall the definition of the anisotropic mixed-norm
Hardy spaces from \cite{cgn17}. For this purpose, we first present the
notions of both anisotropic homogeneous quasi-norms and mixed-norm Lebesgue spaces.

For any $\az:=(\az_1,\ldots,\az_n)$, $x:=(x_1,\ldots,x_n)\in \rn$ and $t\in[0,\fz)$,
let $t^\az x:=(t^{\az_1}x_1,\ldots,t^{\az_n}x_n)$.
The following notion of anisotropic homogeneous quasi-norms
is from \cite{f66} (see also \cite{sw78}).

\begin{definition}\label{2d1}
Let $\va:=(a_1,\ldots,a_n)\in [1,\fz)^n$.
The \emph{anisotropic homogeneous quasi-norm} $|\cdot|_{\va}$,
associated with $\va$, is a non-negative measurable function on $\rn$
defined by setting $|\vec0_n|_{\va}:=0$ and, for any $x\in \rn\setminus\{\vec0_n\}$,
$|x|_{\va}:=t_0$, where $t_0$ is the unique positive number such that $|t_0^{-\va}x|=1$,
namely,
$$\f{x_1^2}{t_0^{2a_1}}+\cdots+\f{x_n^2}{t_0^{2a_n}}=1.$$
\end{definition}

\begin{remark}\label{2r1}
Let $\va\in[1,\fz)^n$. From \cite[Lemma 2.5(i) and (ii)]{hlyy},
it follows that, for any $t\in[0,\fz)$ and $x,\,y\in\rn$,
\begin{align}\label{2e1}
|x+y|_{\va}\le |x|_{\va}+|y|_{\va}
\quad{\rm and}\quad
\lf|t^{\va}x\r|_{\va}= t|x|_{\va},
\end{align}
which implies that $|\cdot|_{\va}$ is a norm if and only
if $\va:=(\overbrace{1,\ldots,1}^{n\ \rm times})$ and, in this case,
the homogeneous quasi-norm $|\cdot|_{\va}$ becomes the Euclidean norm $|\cdot|$.
\end{remark}

Now we recall the notions of the anisotropic bracket
and the homogeneous dimension from \cite{sw78},
which play a key role in the study on anisotropic function spaces.

\begin{definition}\label{2d2}
Let $\va:=(a_1,\ldots,a_n)\in [1,\fz)^n$. The \emph{anisotropic bracket},
associated with $\va$, is defined by setting, for any $x\in \rn$,
$$\langle x\rg_{\va}:=|(1,x)|_{(1,\va)}.$$
Furthermore, the \emph{homogeneous dimension} $\nu$ is defined as
$$\nu:=|\vec{a}|:=a_1+\cdots +a_n.$$
\end{definition}

For any $\va:=(a_1,\ldots,a_n)\in [1,\fz)^n$, let
\begin{align}\label{2e9}
a_-:=\min\{a_1,\ldots,a_n\}\hspace{0.35cm}
{\rm and}\hspace{0.35cm} a_+:=\max\{a_1,\ldots,a_n\}.
\end{align}

For any $\va\in [1,\fz)^n$, $r\in (0,\fz)$ and $x\in \rn$,
the \emph{anisotropic ball} $B_{\va}(x,r)$,
with center $x$ and radius $r$, is defined as
$B_{\va}(x,r):=\{y\in \rn:\,|y-x|_{\va} < r\}$.
Then \eqref{2e1} implies that $B_{\va}(x,r)= x+r^{\va}B_{\va}(\vec0_n,1)$ and
$|B_{\va}(x,r)|=\nu_n r^{\nu}$, where $\nu_n:=|B_{\va}(\vec0_n,1)|$
(see \cite[(2.12)]{cgn17}). Moreover, by \cite[Lemma 2.4(ii)]{hlyy}, we know that 
$B_0:=B_{\va}(\vec0_n,1)=B(\vec0_n,1)$, where $B(\vec0_n,1)$ denotes
the \emph{unit ball} of $\rn$, namely, $B(\vec0_n,1):=\{y\in \rn:\,|y|<1\}$.
For any $t\in (0,\fz)$, let
\begin{align}\label{2e2'}
B^{(t)}:=t^{\va} B_0=B_{\va}(\vec0_n,t).
\end{align}
Throughout this article, the \emph{symbol} $\mathfrak{B}$ always denotes the set of
all anisotropic balls, namely,
\begin{align}\label{2e2}
\mathfrak{B}:=\lf\{x+B^{(t)}:\ x\in\rn,\ t\in(0,\fz)\r\}.
\end{align}

Recall that, for any $r\in(0,\fz]$ and measurable set $E\subset\rn$,
the \emph{Lebesgue space} $L^r(E)$ is defined to be the set of all measurable functions $f$ such that
$$\|f\|_{L^r(E)}:=\lf[\int_E|f(x)|^r\,dx\r]^{1/r}<\fz$$
with the usual modification made when $r=\fz$. Then
we present the following notion of mixed-norm Lebesgue spaces from \cite{bp61}.

\begin{definition}\label{2d3}
Let $\vp:=(p_1,\ldots,p_n)\in (0,\fz]^n$. The \emph{mixed-norm Lebesgue space} $\lv$ is
defined to be the set of all measurable functions $f$ such that
$$\|f\|_{\lv}:=\lf\{\int_{\rr}\cdots\lf[\int_{\rr}\lf\{\int_{\rr}|f(x_1,\ldots,x_n)|^{p_1}
\,dx_1\r\}^{\f{p_2}{p_1}}\,dx_2\r]^{\f{p_3}{p_2}}\cdots\, dx_n\r\}^{\f{1}{p_n}}<\fz$$
with the usual modifications made when $p_i=\fz$ for some $i\in \{1,\ldots,n\}$.
\end{definition}

\begin{remark}\label{2r2}
For any $\vp\in(0,\fz]^n$, $(\lv,\|\cdot\|_{\lv})$ is a quasi-Banach
space and, for any $\vp \in [1,\fz]^n$, $(\lv,\|\cdot\|_{\lv})$ becomes a Banach space
(see \cite[p.\,304, Theorem 1]{bp61}). Obviously, when
$\vp:=(\overbrace{p,\ldots,p}^{n\ \rm times})$ with $p\in(0,\fz]^n$,
$\lv$ coincides with the classical Lebesgue space $L^p(\rn)$.
\end{remark}

For any $\vp:=(p_1,\ldots,p_n)\in (0,\fz)^n$, let
\begin{align}\label{2e10}
p_-:=\min\{p_1,\ldots,p_n\},\hspace{0.35cm}
p_+:=\max\{p_1,\ldots,p_n\}
\hspace{0.35cm}{\rm and}\hspace{0.35cm}
\underline{p}:=\min\{p_-,1\}.
\end{align}

A $C^\infty(\rn)$ function $\varphi$ is called a \emph{Schwartz function} if,
for any $N\in\zz_+$ and multi-index $\az\in\zz_+^n$,
$$\|\varphi\|_{N,\alpha}:=
\sup_{x\in\rn}\lf\{(1+|x|)^N
|\partial^\alpha\varphi(x)|\r\}<\infty.$$
Denote by
$\cs(\rn)$ the set of all Schwartz functions, equipped
with the topology determined by
$\{\|\cdot\|_{N,\alpha}\}_{N\in\zz_+,\az\in\zz_+^n}$,
and $\cs'(\rn)$ its \emph{dual space}, equipped with the weak-$\ast$ topology.
For any $N\in\mathbb{Z}_+$, let
$$\cs_N(\rn):=\lf\{\varphi\in\cs(\rn):\
\|\varphi\|_{\cs_N(\rn)}:=
\sup_{x\in\rn}\lf[\langle x\rg_{\va}^N\sup_{|\az|\le N}
|\partial^\alpha\varphi(x)|\r]\le 1\r\}.$$
In what follows, for any $\varphi \in \cs(\rn)$ and $t\in (0,\fz)$,
let $\varphi_t(\cdot):=t^{-\nu}\varphi(t^{-\va}\cdot)$.

\begin{definition}\label{2d4}
Let $\phi\in\cs(\rn)$ and $f\in\cs'(\rn)$. The
\emph{non-tangential maximal function} $M_\phi(f)$,
with respect to $\phi$, is defined by setting, for any $x\in\rn$,
$$
M_\phi(f)(x):= \sup_{y\in B_{\va}(x,t),
t\in (0,\fz)}|f\ast\phi_t(y)|.
$$
Moreover, for any given $N\in\mathbb{N}$, the
\emph{non-tangential grand maximal function} $M_N(f)$ of
$f\in\cs'(\rn)$ is defined by setting, for any $x\in\rn$,
\begin{equation*}
M_N(f)(x):=\sup_{\phi\in\cs_N(\rn)}
M_\phi(f)(x).
\end{equation*}
\end{definition}

The following anisotropic mixed-norm Hardy space was first introduced
in \cite[Definition 3.3]{cgn17}.

\begin{definition}\label{2d5}
Let $\va\in[1,\fz)^n,\,\vp\in(0,\fz)^n$,
$N_{\vp}:=\lfloor\nu\frac{a_+}{a_-}
(\f{1}{\underline{p}}+1)+\nu+2a_+\rfloor+1$ and
\begin{align}\label{2e11}
N\in\mathbb{N}\cap \lf[N_{\vp},\fz\r),
\end{align}
where $a_-,\,a_+$ are as in \eqref{2e9} and $\underline{p}$ is as in \eqref{2e10}.
The \emph{anisotropic mixed-norm Hardy space} $\vh$ is defined by setting
\begin{equation*}
\vh:=\lf\{f\in\cs'(\rn):\ M_N(f)\in\lv\r\}
\end{equation*}
and, for any $f\in\vh$, let
$\|f\|_{\vh}:=\| M_N(f)\|_{\lv}$.
\end{definition}

\begin{remark}\label{2r3}
\begin{enumerate}
\item[{\rm (i)}] When
$\va:=(\overbrace{1,\ldots,1}^{n\ \rm times})$
and $\vp:=(\overbrace{p,\ldots,p}^{n\ \rm times})$, where $p\in(0,\fz)$,
then, by Remark \ref{2r2}, we know that $\vh$ coincides with the classical
isotropic Hardy space $H^p(\rn)$ of Fefferman and Stein \cite{fs72}.

\item[{\rm (ii)}] The quasi-norm of $\vh$ in Definition \ref{2d5} depends on $N$, however,
the space $\vh$ is independent of the choice of $N$ as long as $N$ is as in \eqref{2e11}
(see \cite[Remark 2.12]{hlyy}).
\end{enumerate}
\end{remark}

\section{Dual space of $\vh$\label{s3}}

Let $\va\in [1,\fz)^n$ and $\vp\in(0,1]^n$.
In this section, we prove that the dual space of $\vh$ is the anisotropic mixed-norm
Campanato space $\mathcal{L}^{\va}_{\vec{p},\,r',\,s}(\mathbb{R}^n)$ with
$r\in(1,\fz]$ and $s$ as in \eqref{3e1} below.
To this end, we first introduce the anisotropic mixed-norm
Campanato space $\lq$.
In what follows, for any given $s\in \mathbb{Z}_+$,
the \emph{symbol $\cp_s(\rn)$} denotes the linear space of all polynomials
on $\rn$ with degree not greater than $s$.

\begin{definition}\label{3d1'}
Let $\va\in [1,\fz)^n$, $\vp\in(0,\fz]^n,\,q\in[1,\fz]$ and $s\in\zz_+$.
The \emph{anisotropic mixed-norm Campanato space} $\lq$ is defined to be
the set of all measurable functions $g$ such that, when $q\in[1,\fz)$,
$$\|g\|_{\lq}:=\sup_{B\in\mathfrak{B}}\inf_{P\in\cp_s(\rn)}
\frac{|B|}{\|\chi_B\|_{\lv}}\lf[\frac1{|B|}\int_B\lf|g(x)-P(x)\r|^q\,dx\r]^{1/q}<\fz$$
and
$$\|g\|_{\mathcal{L}^{\va}_{\vec{p},\,\fz,\,s}(\rn)}
:=\sup_{B\in\mathfrak{B}}\inf_{P\in\cp_s(\rn)}
\frac{|B|}{\|\chi_B\|_{\lv}}\lf\|g-P\r\|_{L^{\fz}(B)}<\fz,$$
where $\mathfrak{B}$ is as in \eqref{2e2}.
\end{definition}

\begin{remark}\label{3r1}
\begin{enumerate}
\item[{\rm (i)}] It is easy to see that $\|\cdot\|_{\lq}$ is a seminorm and
$\cp_s(\rn)\st \lq$. Indeed, $\|g\|_{\lq}=0$ if and only if $g\in\cp_s(\rn)$.
Thus, if we identify $g_1$ with $g_2$ when $g_1-g_2\in\cp_s(\rn)$, then
$\lq$ becomes a Banach space. Throughout this article, we identify
$g\in\lq$ with $\{g+P:\, P\in\cp_s(\rn)\}$.
\item[{\rm (ii)}] When $\va:=(\overbrace{1,\ldots,1}^{n\ \rm times})$ and
$\vp:=(\overbrace{p,\ldots,p}^{n\ \rm times})$ with some $p\in(0,1]$, for any
$B\in\mathfrak{B}$, $\|\chi_B\|_{\lv}=|B|^{1/p}$. In this case, the
space $\lq$ is just the classical Campanato space $L_{\frac1p-1,\,q,\,s}(\rn)$
introduced by Campanato in \cite{c64}, which includes the classical
space $\mathop{\mathrm{BMO}}(\rn)$ of John and Nirenberg \cite{jn61} as a special case.
\end{enumerate}
\end{remark}

The following definitions of anisotropic mixed-norm
$(\vp,r,s)$-atoms, anisotropic mixed-norm atomic Hardy
spaces and anisotropic mixed-norm finite atomic Hardy
spaces are just \cite[Definitions 3.1, 3.2 and 5.1]{hlyy}, respectively.

\begin{definition}\label{3d1}
Let $\va\in [1,\fz)^n$, $\vp\in(0,\fz)^n$, $r\in (1,\fz]$ and
\begin{align}\label{3e1}
s\in\lf[\lf\lfloor\f{\nu}{a_-}\lf(\f{1}{p_-}-1\r) \r\rfloor,\fz\r)\cap\zz_+,
\end{align}
where $a_-$ and $p_-$ are, respectively, as in \eqref{2e9} and \eqref{2e10}.
An \emph{anisotropic mixed-norm $(\vp,r,s)$-atom} $a$ is
a measurable function on $\rn$ satisfying
\begin{enumerate}
\item[{\rm (i)}] $\supp a \st B$, where
$B\in\mathfrak{B}$ with $\mathfrak{B}$ as in \eqref{2e2};

\item[{\rm (ii)}] $\|a\|_{L^r(\rn)}\le \frac{|B|^{1/r}}{\|\chi_B\|_{\lv}}$;

\item[{\rm (iii)}] $\int_{\mathbb R^n}a(x)x^\az\,dx=0$ for any $\az\in\zz_+^n$
with $|\az|\le s$.
\end{enumerate}
\end{definition}

\begin{definition}\label{3d2}
Let $\va$, $\vp$, $r$
and $s$ be as in Definition \ref{3d1}. The \emph{anisotropic mixed-norm
atomic Hardy space} $\vah$ is defined to be the
set of all $f\in\cs'(\rn)$ satisfying that there exist
$\{\lz_i\}_{i\in\nn}\st\mathbb{C}$
and a sequence of $(\vp,r,s)$-atoms, $\{a_i\}_{i\in\nn}$,
supported, respectively, on
$\{B_i\}_{i\in\nn}\st\mathfrak{B}$ such that
\begin{align}\label{3e14}
f=\sum_{i\in\nn}\lz_ia_i
\quad\mathrm{in}\quad\cs'(\rn).
\end{align}
Moreover, for any $f\in\vah$, let
\begin{align*}
\|f\|_{\vah}:=
{\inf}\lf\|\lf\{\sum_{i\in\nn}
\lf[\frac{|\lz_i|\chi_{B_i}}{\|\chi_{B_i}\|_{\lv}}\r]^
{\underline{p}}\r\}^{1/{\underline{p}}}\r\|_{\lv},
\end{align*}
where $\underline{p}$ is as in \eqref{2e10} and the infimum is taken
over all decompositions of $f$ as in \eqref{3e14}.
\end{definition}

\begin{definition}\label{5d1}
Let $\va$, $\vp$, $r$
and $s$ be as in Definition \ref{3d1}. The \emph{anisotropic mixed-norm
finite atomic Hardy space} $\vfah$ is defined to be the set of all
$f\in\cs'(\rn)$ satisfying that there exist $I\in\nn$,
$\{\lz_i\}_{i\in[1,I]\cap\nn}\st\mathbb{C}$ and
a finite sequence of $(\vp,r,s)$-atoms,
$\{a_i\}_{i\in[1,I]\cap\nn}$, supported, respectively, on
$\{B_i\}_{i\in[1,I]\cap\nn}\st\mathfrak{B}$
such that
\begin{align}\label{3e15}
f=\sum_{i=1}^I\lambda_ia_i
\quad\mathrm{in}\quad\cs'(\rn).
\end{align}
Moreover, for any $f\in\vfah$, let
\begin{align*}
\|f\|_{\vfah}:=
{\inf}\lf\|\lf\{\sum_{i=1}^{I}
\lf[\frac{|\lz_i|\chi_{B_i}}{\|\chi_{B_i}\|_{\lv}}\r]^
{\underline{p}}\r\}^{1/\underline{p}}\r\|_{\lv},
\end{align*}
where $\underline{p}$ is as in \eqref{2e10} and the
infimum is taken over all decompositions of $f$ as in \eqref{3e15}.
\end{definition}

To establish the duality theory of $\vh$, we need the following atomic and
finite atomic characterizations of $\vh$,
which are just \cite[Theorems 3.15 and 5.9]{hlyy}, respectively.

\begin{lemma}\label{3l1}
Let $\va$, $\vp$
and $s$ be as in Definition \ref{3d1}, $r\in(\max\{p_+,1\},\fz]$
with $p_+$ as in \eqref{2e10} and $N$ be as in \eqref{2e11}.
Then $\vh=\vah$ with equivalent quasi-norms.
\end{lemma}

\begin{lemma}\label{3l2}
Let $\va$, $\vp$, $r$
and $s$ be as in Lemma \ref{3l1} and
$C(\rn)$ denote the set of all continuous functions on $\rn$.
\begin{enumerate}
\item[{\rm (i)}]
If $r\in(\max\{p_+,1\},\fz)$, then $\|\cdot\|_{\vfah}$
and $\|\cdot\|_{\vh}$ are equivalent quasi-norms on $\vfah$;

\item[{\rm (ii)}]
$\|\cdot\|_{\vfahfz}$
and $\|\cdot\|_{\vh}$ are equivalent quasi-norms on
$\vfahfz\cap C(\rn)$.
\end{enumerate}
\end{lemma}

Via borrowing some ideas from the proofs of \cite[Lemma 5.9]{zsy16}
and \cite[p.\,51, Lemma 8.2]{mb03}, respectively, we obtain the following two lemmas.

\begin{lemma}\label{3l3}
Let $\vp\in (0,1]^n$. Then, for any $\{\lz_i\}_{i\in\nn}\st\mathbb{C}$
and $\{B_i\}_{i\in\nn}\st\mathfrak{B}$,
$$\sum_{i\in\nn}|\lz_i|\le \lf\|\lf\{\sum_{i\in\nn}
\lf[\frac{|\lz_i|\chi_{B_i}}{\|\chi_{B_i}\|_{\lv}}\r]^
{\underline{p}}\r\}^{1/{\underline{p}}}\r\|_{\lv},$$
where $\underline{p}$ is as in \eqref{2e10}.
\end{lemma}

\begin{proof}
Let $\lz:=\sum_{i\in\nn}|\lz_i|$. Notice that,
for any $\{\lz_i\}_{i\in\nn}\st\mathbb{C}$ and $\theta\in(0,1]$,
$$\lf[\sum_{i\in\nn}|\lz_i|\r]^{\theta}\le \sum_{i\in\nn}|\lz_i|^{\theta}.$$
By the fact that $\vp\in (0,1]^n$ and \eqref{2e10}, we find that
\begin{align*}
\lf\|\lf\{\sum_{i\in\nn}
\lf[\frac{|\lz_i|\chi_{B_i}}{\lz\|\chi_{B_i}\|_{\lv}}\r]^
{\underline{p}}\r\}^{1/{\underline{p}}}\r\|_{\lv}
\ge\lf\|\sum_{i\in\nn}
\frac{|\lz_i|\chi_{B_i}}{\lz\|\chi_{B_i}\|_{\lv}}\r\|_{\lv}
\ge\sum_{i\in\nn}\f{|\lz_i|}{\lz}\lf\|
\frac{\chi_{B_i}}{\|\chi_{B_i}\|_{\lv}}\r\|_{\lv}=1,
\end{align*}
which implies the desired conclusion and hence completes the proof of Lemma \ref{3l3}.
\end{proof}

\begin{lemma}\label{3l4}
Let $\vp\in (0,1]^n$ and $\va$, $r$
and $s$ be as in Definition \ref{3d1}. Then, for any continuous
linear functional $L$ on $\vh=\vah$,
\begin{align}\label{3e3}
\lf\|L\r\|_{[\vah]^*}:=\sup\lf\{|L(f)|:\,\|f\|_{\vah}\le 1\r\}
=\sup\lf\{|L(a)|:\,a\ is\ any\ (\vp,r,s){\text-}atom\r\},
\end{align}
here and hereafter, $[\vah]^*$ denotes the dual space of $\vah$.
\end{lemma}

\begin{proof}
For any $(\vp,r,s)$-atom $a$, we easily know that $\|a\|_{\vah}\le 1$.
Thus,
\begin{align}\label{3e2}
\sup\lf\{|L(a)|:\,a\ {\rm is\ any}\ (\vp,r,s){\text-}{\rm atom}\r\}
\le \sup\lf\{|L(f)|:\,\|f\|_{\vah}\le 1\r\}.
\end{align}

Conversely, let $f\in\vh$ and $\|f\|_{\vah}\le 1$.
Then, for any $\varepsilon\in(0,\fz)$, by an argument similar to that used
in the proof of \cite[Theorem 3.15]{hlyy},
we conclude that there exist $\{\lz_i\}_{i\in\nn}\st\mathbb{C}$
and a sequence of $(\vp,r,s)$-atoms, $\{a_i\}_{i\in\nn}$,
supported, respectively, on
$\{B_i\}_{i\in\nn}\st\mathfrak{B}$ such that
$$
f=\sum_{i\in\nn}\lz_ia_i
~~{\rm in}~~\vh
\quad {\rm and}\quad
\lf\|\lf\{\sum_{i\in\nn}
\lf[\frac{|\lz_i|\chi_{B_i}}{\|\chi_{B_i}\|_{\lv}}\r]^
{\underline{p}}\r\}^{1/{\underline{p}}}\r\|_{\lv}\le 1+\varepsilon.
$$
From this, the continuity of $L$ and Lemma \ref{3l3},
we further deduce that
\begin{align*}
|L(f)|&\le \sum_{i\in\nn}|\lz_i||L(a_i)|
\le\lf[\sum_{i\in\nn}|\lz_i|\r]
\sup\lf\{|L(a)|:\,a\ {\rm is\ any}\ (\vp,r,s){\text-}{\rm atom}\r\}\\
&\le(1+\varepsilon)\sup\lf\{|L(a)|:\,a\ {\rm is\ any}\ (\vp,r,s){\text-}{\rm atom}\r\},
\end{align*}
which, combined with the arbitrariness of $\varepsilon\in(0,\fz)$
and \eqref{3e2}, implies that \eqref{3e3} holds true.
This finishes the proof of Lemma \ref{3l4}.
\end{proof}

The main result of this section is stated as follows.

\begin{theorem}\label{3t1}
Let $\va,\,\vp,\,r$ and $s$ be as in Lemma \ref{3l4}.
Then the dual space of $\vh$, denoted by $[\vh]^*$, is $\lr$ in the following sense:
\begin{enumerate}
\item[{\rm (i)}] Suppose that $g\in\lr$. Then the linear functional
$$L_g:\,f\longmapsto L_g(f):=\int_{\rn}f(x)g(x)\,dx,$$
initially defined for any $f\in\vfah$ has a bounded extension to $\vh$.

\item[{\rm (ii)}] Conversely, any continuous linear functional on $\vh$
arises as in (i) with a unique $g\in\lr$.
\end{enumerate}
Moreover, $\|g\|_{\lr}\sim \|L_g\|_{[\vh]^*}$, where the implicit equivalent
positive constants are independent of $g$.
\end{theorem}

\begin{remark}\label{3r2}
\begin{enumerate}
\item[{\rm (i)}]
When $\va$ and $\vp$ are as in Remark \ref{3r1}(ii),
$\vh$ and $\lr$ become, respectively, the classical Hardy space $H^p(\rn)$
and Campanato space $L_{\frac1p-1,\,r',\,s}(\rn)$ (see \cite{c64}).
In this case,
Theorem \ref{3t1} was proved by Taibleson and Weiss \cite{tw80}, which includes
the famous duality result of Fefferman and Stein \cite{fs72}, namely,
$[H^1(\rn)]^*={\mathop{\mathrm{BMO}}}(\rn)$, as a special case.
\item[{\rm (ii)}] We should point out that, when $\va$ is as in Remark \ref{3r1}(ii),
the space $\vh$ is just the isotropic mixed-norm Hardy space. Even in this
case, Theorem \ref{3t1} is also new.
\item[{\rm (iii)}]
When $\vp\in(1,\fz)^n$, it was proved in \cite[Theorem 6.1]{cgn17} that
$\vh=\lv$ with equivalent quasi-norms. This, together with
\cite[p.\,304, Theorem 1.a)]{bp61}, implies that, for any $\vp\in(1,\fz)^n$,
$L^{\vp'}(\rn)$ is the dual space of $\vh$.
However, when
$\vp:=(p_1,\ldots,p_n)\in(0,\fz)^n$ with $p_{i_0}\in(0,1]$ and $p_{j_0}\in(1,\fz)$
for some $i_0,j_0\in\{1,\ldots,n\}$, the dual space of $\vh$ is still unknown so far.
\end{enumerate}
\end{remark}

As an immediate corollary of Theorem \ref{3t1}, we have the following
equivalence of the spaces $\lq$, the details being omitted.

\begin{corollary}
Let $\va,\,\vp$ and $s$ be as in Theorem \ref{3t1} and $q\in[1,\fz)$.
Then $\mathcal{L}^{\va}_{\vec{p},\,1,\,s}(\rn)=\lq$ with equivalent quasi-norms.
\end{corollary}

Now we prove Theorem \ref{3t1}.

\begin{proof}[Proof of Theorem \ref{3t1}]
By Lemma \ref{3l1}, to prove $\lr\st [\vh]^*$, it suffices to show
$$\lr\st [\vah]^*.$$
To this end, let $g\in \lr$ and $a$ be a $(\vp,r,s)$-atom supported on $B\st\mathfrak{B}$.
Then, from Definition \ref{3d1}, the H\"{o}lder inequality and Definition \ref{3d1'},
it follows that
\begin{align*}
\lf|\int_{\rn}a(x)g(x)\,dx\r|
&=\inf_{P\in\cp_s(\rn)}\lf|\int_{\rn}a(x)\lf[g(x)-P(x)\r]\,dx\r|\\
&\le \|a\|_{L^r(\rn)}\inf_{P\in\cp_s(\rn)}\lf[\int_{\rn}\lf|g(x)-P(x)\r|^{r'}\,dx\r]^{1/r'}\noz\\
&\le \frac{|B|^{1/r}}{\|\chi_B\|_{\lv}}\inf_{P\in\cp_s(\rn)}
\lf[\int_{\rn}\lf|g(x)-P(x)\r|^{r'}\,dx\r]^{1/r'}\le \|g\|_{\lr}\noz.
\end{align*}
By this and Lemma \ref{3l3}, we find that, for any $m\in\nn$,
$\{\lz_i\}_{i=1}^{m}\st\mathbb{C}$,
a sequence $\{a_i\}_{i=1}^{m}$ of $(\vp,r,s)$-atoms
supported, respectively, on $\{B_i\}_{i=1}^{m}\st\mathfrak{B}$
and $f=\sum_{i=1}^m\lambda_ia_i\in \vfah$,
\begin{align*}
\lf|L_g(f)\r|&=\lf|\int_{\rn}f(x)g(x)\,dx\r|
\le \sum_{i=1}^m|\lambda_i|\int_{\rn}|a_i(x)g(x)|\,dx\\
&\le \sum_{i=1}^m|\lambda_i|\|g\|_{\lr}
\le \|f\|_{\vfah}\|g\|_{\lr},
\end{align*}
which, together with the fact that $\vfah$ is dense in
$\vah$ and Lemma \ref{3l2}, implies that (i) holds true.

Conversely, for any $B\in\mathfrak{B}$, let
$$\Pi_B:\ L^1(B)\longrightarrow \cp_s(\rn)$$
be the natural projection satisfying,
for any $f\in L^1(B)$ and $q\in\cp_s(\rn)$,
\begin{align}\label{3e5}
\int_B\Pi_B(f)(x)q(x)\,dx=\int_Bf(x)q(x)\,dx.
\end{align}
Then there exists a positive constant $C_{(s)}$, depending on $s$, such that,
for any $B\in\mathfrak{B}$ and $f\in L^1(B)$,
\begin{align}\label{3e6}
\sup_{x\in B}\lf|\Pi_B(f)(x)\r|\le C_{(s)}\frac1{|B|}\int_B|f(y)|\,dy.
\end{align}
Indeed, if $B:=B_0$, then one may find an orthonormal basis
$\{q_{\az}\}_{|\az|\le s}$ of $\cp_s(\rn)$ with respect to the
$L^2(B_0)$ norm. By \eqref{3e5}, we know that, for any $f\in L^1(B_0)$,
$$\Pi_{B_0}(f)
=\sum_{|\az|\le s}\lf[\int_{B_0}\Pi_{B_0}(f)(y)\overline{q_{\az}(y)}\,dy\r]q_{\az}
=\sum_{|\az|\le s}\lf[\int_{B_0}f(y)\overline{q_{\az}(y)}\,dy\r]q_{\az}.$$
Thus, there exists some $\az_0\in\zz_+^n$ with $|\az_0|\le s$, such that
\begin{align}\label{3e7}
\sup_{x\in B_0}\lf|\Pi_{B_0}(f)(x)\r|
\ls\sup_{x\in B_0}\lf\{\lf[\int_{B_0}|f(y)||q_{\az_0}(y)|\,dy\r]|q_{\az_0}(x)|\r\}
\ls\frac1{|B_0|}\int_{B_0}|f(y)|\,dy,
\end{align}
which implies that \eqref{3e6} holds true for $B:=B_0$. In addition, for any
$\ell\in(0,\fz)$ and $f\in L^1(\rn)$,
let $D_{\ell^{\va}}(f)(\cdot):=\ell^{\nu}f(\ell^{\va}\cdot)$.
Then, from \eqref{3e5}, we deduce that, for any $\ell\in(0,\fz)$,
$f\in L^1(B^{(\ell)})$ with $B^{(\ell)}$ as in \eqref{2e2'} and $q\in\cp_s(\rn)$,
\begin{align*}
\int_{B^{(\ell)}}\lf(D_{\ell^{-\va}}\circ\Pi_{B_0}\circ D_{\ell^{\va}}\r)(f)(x)q(x)\,dx
&=\ell^{-\nu}\int_{B^{(\ell)}}\Pi_{B_0}\lf(D_{\ell^{\va}}(f)\r)(\ell^{-\va}x)q(x)\,dx\\
&=\int_{B_0}D_{\ell^{\va}}(f)(y)q(\ell^{\va}y)\,dy
=\int_{B^{(\ell)}}f(x)q(x)\,dx,
\end{align*}
which implies that
$\Pi_{B^{(\ell)}}(f)=(D_{\ell^{-\va}}\circ\Pi_{B_0}\circ D_{\ell^{\va}})(f)$.
Therefore, by \eqref{3e7}, we find that, for any $\ell\in(0,\fz)$ and
$f\in L^1(B^{(\ell)})$,
\begin{align*}
\sup_{x\in B^{(\ell)}}\lf|\Pi_{B^{(\ell)}}(f)(x)\r|
&=\ell^{-\nu}\sup_{x\in B^{(\ell)}}\lf|\Pi_{B_0}\lf(D_{\ell^{\va}}(f)\r)(\ell^{-\va}x)\r|\\
&\ls\ell^{-\nu}\frac1{|B_0|}\int_{B_0}\lf|D_{\ell^{\va}}(f)(y)\r|\,dy
\sim \frac1{|B^{(\ell)}|}\int_{B^{(\ell)}}|f(y)|\,dy.
\end{align*}
Thus, for any $\ell\in(0,\fz)$, \eqref{3e6} holds true for $B:=B^{(\ell)}$.
Similarly, since,
for any $\ell\in(0,\fz),\,z\in\rn$ and $f\in L^1(z+B^{(\ell)})$,
$\Pi_{z+B^{(\ell)}}(f)=(\tau_{z}\circ\Pi_{B^{(\ell)}}\circ\tau_{-z})(f)$,
where $\tau_{z}(f)(\cdot):=f(\cdot-z)$, it follows that \eqref{3e6} holds true
for any $z+B^{(\ell)}$ with $z\in\rn$ and $\ell\in(0,\fz)$.
This proves \eqref{3e6}.

For any $r\in(1,\fz]$ and $B\in\mathfrak{B}$,
let $L^r_0(B):=\{f\in L^r(B):\,\Pi_B(f)=0\}$.
Then $L^r_0(B)$ is a closed subspace of $L^r(B)$, where one should identify $L^r(B)$
with all the $L^r(\rn)$ functions vanishing outside $B$. With this identification,
for any $f\in L^r_0(B)$,
$$a:=\frac{|B|^{1/r}}{\|\chi_B\|_{\lv}}\|f\|_{L^r(B)}^{-1}f$$
is a $(\vp,r,s)$-atom. By this and Lemma \ref{3l4},
we easily know that, for any $L\in[\vh]^*=[\vah]^*$ and $f\in L^r_0(B)$,
\begin{align}\label{3e8}
|L(f)|\le \frac{\|\chi_B\|_{\lv}}{|B|^{1/r}}\|L\|_{[\vah]^*}\|f\|_{L^r(B)}.
\end{align}
Therefore, $L$ is a bounded linear functional on $L^r_0(B)$ which, by the
Hahn-Banach theorem (see, for example, \cite[Theorem 3.6]{ru91}),
can be extended to the space $L^r(B)$ without increasing its norm.
When $r\in(1,\fz)$, by the duality $[L^r(B)]^*=L^{r'}(B)$, where $1/r+1/r'=1$,
we know that there exists an $h\in L^{r'}(B)$ such that, for any $f\in L^r_0(B)$,
\begin{align}\label{3e9}
L(f)=\int_B f(x)h(x)\,dx.
\end{align}
When $r=\fz$, from the fact that $L^{\fz}_0(B)\st L^{\wz r}(B)$ with $\wz r\in[1,\fz)$
and the Hahn-Banach theorem again, we deduce that the bounded linear functional $L$
on $L^{\fz}_0(B)$ can be extended to $L^{\wz r}(B)$ without increasing its norm.
By this and \eqref{3e9}, we further conclude that
there exists some $h\in L^{\wz r'}(B)\st L^1(B)$ such that, for any $f\in L^{\fz}_0(B)$,
\eqref{3e9} also holds true. Thus, for any $r\in(1,\fz]$,
there exists an $h\in L^{r'}(B)$ such that,
for any $f\in L^{r}_0(B)$, \eqref{3e9} holds true.

Let $r\in(1,\fz]$.
Next we show that, if there exists another function $\wz h\in L^{r'}(B)$ such that,
for any $f\in L^{r}_0(B)$, $L(f)=\int_B f(x)\wz h(x)\,dx$, then $h-\wz h\in\cp_s(B)$,
where $\cp_s(B)$ denotes all the $\cp_s(\rn)$ elements vanishing outside $B$.
To this end, it suffices to show that, if $h,\,\wz h\in L^{1}(B)$ such that,
for any $f\in L^{\fz}_0(B)$, $\int_B f(x)h(x)\,dx=\int_B f(x)\wz h(x)\,dx$,
then $h-\wz h\in\cp_s(B)$. Indeed, for any $f\in L^{\fz}_0(B)$, we have
\begin{align}\label{3e12}
0&=\int_B \lf[f(x)-\Pi_B(f)(x)\r]\lf[h(x)-\wz h(x)\r]\,dx\\
&=\int_B f(x)\lf[h(x)-\wz h(x)\r]\,dx-\int_B \Pi_B(f)(x)\Pi_B\lf(h-\wz h\r)(x)\,dx\noz\\
&=\int_B f(x)\lf[h(x)-\wz h(x)\r]\,dx-\int_B f(x)\Pi_B\lf(h-\wz h\r)(x)\,dx\noz\\
&=\int_B f(x)\lf[h(x)-\wz h(x)-\Pi_B\lf(h-\wz h\r)(x)\r]\,dx.\noz
\end{align}
In addition, we claim that, for any $B\in\mathfrak{B}$,
\begin{align}\label{3e13}
L^{\fz}_0(B)=L^{\fz}(B)/{\cp_s(B)}.
\end{align}
Actually, applying \cite[Theorem 1.1]{sawa17} with $X:=L^1(B)$ and
$V:=\cp_s(B)$ and the fact that $\cp_s(B)\st L^1(B)\st[L^{\fz}(B)]^*$,
we easily obtain \eqref{3e13}. By this, we find that,
for any $g\in L^{\fz}(B)$, $\{g+P:\, P\in\cp_s(B)\}\in L^{\fz}_0(B)$.
Thus, by \eqref{3e12} and \eqref{3e5}, we conclude that,
for any $g\in L^{\fz}(B)$,
\begin{align*}
\int_B g(x)\lf[h(x)-\wz h(x)-\Pi_B\lf(h-\wz h\r)(x)\r]\,dx=0,
\end{align*}
which implies that, for almost every $x\in B$,
$h(x)-\wz h(x)=\Pi_B(h-\wz h)(x)$ and hence $h-\wz h\in\cp_s(B)$.
Therefore, for any $r\in(1,\fz]$ and $f\in L^{r}_0(B)$, there exists
a unique $h\in L^{r'}(B)/{\cp_s(B)}$ such that \eqref{3e9} holds true.

Assume $r\in(1,\fz]$.
For any $k\in\nn$ and $f\in L^r_0(B^{(k)})$,
let $g_k\in L^{r'}(B^{(k)})/{\cp_s(B^{(k)})}$
be the unique element such that
\begin{align*}
L(f)=\int_{B^{(k)}} f(x)g_k(x)\,dx,
\end{align*}
where, for any $k\in\nn$, $B^{(k)}$ is as in \eqref{2e2'}.
Then it easy to see that, for any $i,\,k\in\nn$ with $i<k$, $g_k|_{B^{(i)}}=g_i$.
From this and the fact that, for any $f\in\vfah$, there exists some $k_0\in\nn$
such that $f\in L^r_0(B^{(k_0)})$, it follows that, for any $f\in\vfah$,
\begin{align}\label{3e16}
L(f)=\int_{\rn}f(x)g(x)\,dx,
\end{align}
where $g(x):=g_k(x)$ for any $x\in B^{(k)}$ with $k\in\nn$.

Thus, to completes the proof of Theorem \ref{3t1}(ii), it remains to prove that
$g\in\lr$. Indeed,
by \eqref{3e16} and \eqref{3e8}, it is easy to see that, for any $r\in(1,\fz]$
and $B\in\mathfrak{B}$,
\begin{align}\label{3e11}
\|g\|_{[L^{r}_0(B)]^*}\le \frac{\|\chi_B\|_{\lv}}{|B|^{1/r}}\|L\|_{[\vah]^*}.
\end{align}
In addition, by an argument similar to that used in the proof of
\cite[p.\,52, (8.12)]{mb03}, we conclude that, for any $r\in(1,\fz]$
and $B\in\mathfrak{B}$,
\begin{align*}
\|g\|_{[L^{r}_0(B)]^*}=\inf_{P\in\cp_s(\rn)}\|g-P\|_{L^{r'}(B)},
\end{align*}
which, combined with Definition \ref{3d1'} and \eqref{3e11}, further implies that,
for any $r\in(1,\fz]$,
\begin{align*}
\|g\|_{\lr}=\sup_{B\in\mathfrak{B}}\frac{|B|^{1/r}}{\|\chi_B\|_{\lv}}
\inf_{P\in\cp_s(\rn)}\|g-P\|_{L^{r'}(B)}
=\sup_{B\in\mathfrak{B}}\frac{|B|^{1/r}}{\|\chi_B\|_{\lv}}\|g\|_{[L^{r}_0(B)]^*}
\le \|L\|_{[\vah]^*}.
\end{align*}
This finishes the proof of Theorem \ref{3t1}(ii) and hence of Theorem \ref{3t1}.
\end{proof}

\bigskip

\noindent  Long Huang, Jun Liu, Dachun Yang (Corresponding author) and Wen Yuan

\medskip

\noindent  Laboratory of Mathematics and Complex Systems
(Ministry of Education of China),
School of Mathematical Sciences, Beijing Normal University,
Beijing 100875, People's Republic of China

\smallskip

\noindent {\it E-mails}:
\texttt{longhuang@mail.bnu.edu.cn} (L. Huang)

\noindent\phantom{{\it E-mails:}}
\texttt{junliu@mail.bnu.edu.cn} (J. Liu)

\noindent\phantom{{\it E-mails:}}
\texttt{dcyang@bnu.edu.cn} (D. Yang)

\noindent\phantom{{\it E-mails:}}
\texttt{wenyuan@bnu.edu.cn} (W. Yuan)

\end{document}